\documentclass[11pt]{amsart}

\usepackage{parskip}
\usepackage{amssymb}

\usepackage{mathrsfs}
\usepackage{color}

\usepackage{amsmath}
\usepackage{amsthm}

\newtheorem{theorem}{Theorem}[section]
\newcounter{tmp}

\newtheorem{lemma}[theorem]{Lemma}
\newtheorem{proposition}[theorem]{Proposition}
\newtheorem*{lemma*}{Lemma}
\newtheorem{corollary}[theorem]{Corollary}
\newtheorem{conj}[theorem]{Conjecture}

\usepackage{turnstile}
\theoremstyle{definition}
\newtheorem{definition}[theorem]{Definition}
\newtheorem{example}[theorem]{Example}

\theoremstyle{remark} 
\newtheorem{remark}[theorem]{Remark}

\newcommand{\abs}[1]{\lvert#1\rvert}

\newcommand{\so}{\mathcal{O}}

\newcommand{\fQ}{\mathbb{Q}}

\newcommand{\fR}{\mathbb{R}}

\newcommand{\fC}{\mathbb{C}}

\newcommand{\rN}{\mathbb{N}}

\newcommand{\rZ}{\mathbb{Z}}

\newcommand{\doD}{\mathscr{D}}

\newcommand{\diV}{\textrm{div}}

\newcommand{\Vw}{\rho}

\newcommand{\dmM}{\mathcal{M}}

\newcommand{\row}[2]{#1_1,\ldots,#1_{#2}}

\usepackage{hyperref}
\hypersetup{
    colorlinks=true,
    linkcolor=blue,
    filecolor=magenta,      
    urlcolor=cyan,
}
 
\begin{document}

\title[Hodge filtration for weighted homogeneous singular $\fQ$-divisor]{Hodge filtration and Hodge ideals for $\fQ$-divisors with weighted homogeneous isolated singularities}

\author{Mingyi Zhang}
\address{Department of Mathematics, Northwestern University, 2033 Sheridan Road, Evanston, IL 60208, USA}
\email{mingyi@math.northwestern.edu}

\begin{abstract}
We give an explicit formula for the Hodge filtration on the $\doD_X$-module $\so_X(*Z)f^{1-\alpha}$ associated to the effective $\fQ$-divisor $D=\alpha\cdot Z$, where $0<\alpha\le1$ and $Z=(f=0)$ is an irreducible hypersurface defined by $f$, a weighted homogeneous polynomial with an isolated singularity at the origin. In particular this gives a formula for the Hodge ideals of $D$. We deduce a formula for the generating level of the Hodge filtration, as well as further properties of Hodge ideals in this setting. We also extend the main theorem to the case when $f$ is a germ of holomorphic function that is convenient and has non-degenerate Newton boundary.
\end{abstract}
\maketitle

\section{Introduction}
\begingroup
\setcounter{tmp}{\value{theorem}}
\setcounter{theorem}{0}

\renewcommand*{\thetheorem}{\Alph{theorem}}

  Let $D$ be an integral and reduced effective divisor on a smooth complex variety $X$. Let $\so_X(*D)$ be the sheaf of rational functions with poles along $D$. This is also a left $\doD_X$-module underlying the mixed Hodge module $j_*\fQ_U^H[n]$, where $U=X\backslash D$ and $j:U\hookrightarrow X$ is the inclusion map.
Any $\doD_X$-module associated to a mixed Hodge module has a good filtration  $F_\bullet$, the Hodge filtration of the mixed Hodge module \cite{Sai90}.

To study the Hodge filtration of $\so_X(*D)$, it seems more convenient to consider a series of ideal sheaves, defined by Musta\c{t}\u{a} and Popa \cite{MP16}, which can be considered to be a generalization of multiplier ideals of divisors. The Hodge ideals $\{I_k(D)\}_{k\in\rN}$ of the divisor $D$ are defined by:
$$F_k\so_X(*D)=I_k(D)\otimes_{\so_X}\so_X\big{(}(k+1)D\big{)},~\forall~k\in\rN.$$
It turns out that $I_0(D)=\mathscr{J}\big{(}(1-\epsilon)D\big{)}$, the multiplier ideal of the divisor $(1-\epsilon)D$, $0<\epsilon\ll1$.

Recently, in \cite{MP18a} and \cite{MP18b}, the authors extend the notion of Hodge ideals to the case when $D$ is an arbitrary effective $\fQ$-divisor on X. Hodge ideals $\{I_k(D)\}_{k\in\rN}$ are defined in terms of the Hodge filtration $F_\bullet$ on some $\doD_X$-module associated with $D$ (see definition in Section \ref{S:filtereddm}; cf. \cite[\S 2-4]{MP18a} for more details). When $D$ is an integral and reduced divisor, this recovers the definition of Hodge ideals $I_k(D)$ mentioned above.

The Hodge filtration $F_\bullet$ is usually hard to describe. However, it does have an explicit formula in the case when $D$ is defined by a reduced weighted homogeneous polynomial $f$ that has an isolated singularity at the origin, which is proved by Morihiko Saito \cite{Sai09}. To state Saito's result and to clarify the notations for this paper, we denote:
\begin{itemize}
\item $\so=\fC\{x_1,\ldots, x_n\}$ the ring of germs of holomorphic function for local coordinates $x_1,\ldots, x_n$.

\item
$f:(\fC^n,0)\to (\fC,0)$ a germ of holomorphic function that is quasihomogeneous, i.e., $f\in \mathcal{J}(f)=(\frac{\partial f}{\partial x_1},\ldots, \frac{\partial f}{\partial x_n})$, and with isolated singularity at the origin. Kyoji Saito \cite{KSai71} showed that after a biholomorphic coordinate change, we can assume $f$ is a weighted homogeneous polynomial with an isolated singularity at the origin. We will keep this assumption for $f$ unless otherwise stated. 

\item $w=w(f)=(\row wn)$ the weights of the weighted homogeneous polynomial $f$.

\item $g:(\fC^n,0)\to(\fC,0)$ a germ of a holomorphic function, and we write $$g=\sum_{A\in\rN^n}g_Ax^A,$$
where $A=(a_1,\ldots, a_n)$, $g_A\in\fC$ and $x^A=x_1^{a_1}\cdots x_n^{a_n}$. 

\item $\rho(g)$ the weight of an element $g\in\so$ defined by 
\begin{equation}\label{E:wtfunction}
\rho(g)=\left(\sum_{i=1}^mw_i\right)+\inf\{\langle w, A\rangle: g_A\neq 0\}.
\end{equation} 
The weight function $\rho$ defines a filtration on $\so$ as 
$$\so^{>k}=\{u\in\so: \rho(u)>k\};$$
$$\so^{\ge k}=\{u\in\so:\rho(u)\ge k\}.$$

\end{itemize}

Throughout the paper, we consider $\doD_X$-modules locally around the isolated singularity, so we can assume $X=\fC^n$ and identify the stalk at the singularity to be that of $\doD_X$-modules on $\fC^n$. For example, we replace $F_k\so_{X,0}(*D)$ by $F_k\so_X(*D)$.

Now we can state the formula proved by M. Saito (see \cite[Theorem 0.7]{Sai09}), namely:
\begin{equation}\label{E:saito}
F_k\so_X(*D)=\sum_{i=0}^kF_{k-i}\doD_X\left(\frac{\so^{\ge i+1}}{f^{i+1}}\right),~\forall~ k\in\rN.
\end{equation}
Since we can now construct a Hodge filtration on analogous $\doD_X$-modules associated to any effective $\fQ$-divisor $D$, it is natural to ask if it satisfies a similar formula in the case when $D$ is supported on a hypersurface defined by such a polynomial $f$.

From now on, we set the divisor to be $D=\alpha Z$, where $0<\alpha\le 1$ and $Z=(f=0)$ is an integral and reduced effective divisor defined by $f$, a weighted homogeneous polynomial with an isolated singularity at the origin. In this case, the associated $\doD_X$-module is the well-known twisted localization $\doD_X$-module $\dmM(f^{1-\alpha}):=\so_X(*Z)f^{1-\alpha}$ \big{(}see more details in \cite{MP18a} about how to construct the Hodge filtration $F_\bullet\dmM(f^{1-\alpha})$\big{)}. With new ingredients from Musta\c{t}\u{a} and Popa's \cite{MP18b}, where this Hodge filtration is compared to the $V$-filtration on $\dmM(f^{1-\alpha})$, we can generalize Saito's formula and prove the following theorem:
\begin{theorem}\label{thma}
If $D=\alpha Z$, where $0<\alpha\le 1$ and $Z=(f=0)$ is an integral and reduced effective divisor defined by $f$, a weighted homogeneous polynomial with an isolated singularity at the origin, then we have
 $$F_k\dmM(f^{1-\alpha})=\sum_{i=0}^kF_{k-i}\doD_X\cdot\left(\frac{\so^{\ge \alpha+i}}{f^{i+1}}f^{1-\alpha}\right),$$
where the action $\cdot$ of $\doD_X$ on the right hand side is the action on the left $\doD_X$-module $\dmM(f^{1-\alpha})$ defined by
$$D\cdot(wf^{1-\alpha}):=\left(D(w)+w\frac{(1-\alpha)D(f)}{f}\right)f^{1-\alpha}, ~\textrm{for any}~D\in \textrm{Der}_{\fC}\so_X.$$
\end{theorem}
 Notice that if we set $\alpha=1$, Theorem \ref{thma} recovers Saito's formula (\ref{E:saito}) mentioned above.

For any polynomial $f$ with an isolated singularity at the origin, it is well known that the Jacobian algebra
$$\mathcal{A}_f:=\fC\{\row xn\}/(\partial_1f,\ldots, \partial_nf)$$
is a finite dimensional $\fC$-vector space. Fix a monomial basis $\{\row v\mu\}$ for this vector space, where $\mu$ is the dimension of $\mathcal{A}_f$ (usually called the Milnor number of $f$). With this notation, applying Theorem \ref{thma}, we can give a formula for the Hodge filtration that is easier to use in practice:
\begin{corollary}\label{C:B}
If $D=\alpha Z$, where $0<\alpha\le 1$ and $Z=(f=0)$ is an integral and reduced effective divisor defined by $f$, a weighted homogeneous polynomial with an isolated singularity at the origin, then we have
$$F_0\dmM(f^{1-\alpha})=f^{-1}\cdot \so^{\ge \alpha}f^{1-\alpha}$$
and
$$F_k\dmM(f^{1-\alpha})=(f^{-1}\cdot \sum_{v_j\in \so^{\ge k+1+\alpha}}\fC v_j) f^{1-\alpha}+F_1\doD_X\cdot F_{k-1}\dmM(f^{1-\alpha}).$$
Alternatively, in terms of Hodge ideals these formulas say that
$$I_0(D)=\so^{\ge\alpha}$$
and
$$I_{k+1}(D)=\sum_{v_j\in \so^{\ge k+1+\alpha}}\fC v_j+\sum_{1\le i\le n,a\in I_{k}(D)}\mathcal{O}_X\big{(}f\partial_ia-(\alpha+k)a\partial_if\big{)}.$$
\end{corollary}
In particular, the Hodge filtration is fully computable since it is good and hence can be determined by finitely many terms. In order to do this effectively, Saito \cite{Sai09} introduced the following measure of the complexity of the Hodge filtration: The \emph{generating level} of any $\doD_X$-module $(\dmM, F_{\bullet})$ with a good filtration is the smallest integer $k$ such that
$$F_l\doD_X\cdot F_k\dmM=F_{k+l}\dmM\quad\textrm{for all}\quad l\ge0.$$
Note that such a $k$ always exist by definition. In the case $D$ is integral and reduced, defined by a weighted homogeneous polynomial with an isolated singularity at the origin, Saito proves that the generating level of $\so_X(*D)$ is $[n-\tilde{\alpha}_f-1]$, where $\tilde{\alpha}_f$ is the minimal exponent (see \cite[Section 6]{MP18b}) which is also called the microlocal log canonical threshold (see \cite[Theorem 0.7]{Sai09}). Popa conjectured in \cite[Question 5.10]{Pop18} that a similar result should hold in the $\fQ$-divisor setup we considered here. Using Theorem \ref{thma}, we prove this conjecture is true:

\begin{corollary}\label{C:generatinglevel}
If $D=\alpha Z$, where $0<\alpha\le 1$ and $Z=(f=0)$ is an integral and reduced effective divisor defined by $f$, a weighted homogeneous polynomial with an isolated singularity at the origin, then the generating level of $\dmM(f^{1-\alpha})$ is $[n-\tilde{\alpha}_f-\alpha]$.
\end{corollary}

Another application of Theorem \ref{thma} is to prove the inclusion 
\begin{equation}\label{E:inclusion}
I_k(D)\subseteq I_{k-1}(D),~\forall~k\ge 1
\end{equation}
 (see more details in Corollary \ref{C:hodgeidealseq}). Notice that for any reduced integral divisor $Z$, we have 
$$I_k(Z)\subseteq I_{k-1}(Z),~\forall~ k\ge 1$$
(as proved in \cite[Proposition 13.1]{MP16}), but in general it is not clear whether an inclusion like (\ref{E:inclusion}) should hold for arbitrary $\fQ$-divisors (see \cite[Remark 4.2]{MP18a}).

As a consequence of Corollary \ref{C:B}, we also deduce that the roots of the Bernstein-Sato polynomial can be related to the jumping numbers and coefficients associated to ideals of the form $I_k(cZ)/(\partial f)$ for varking $k$ and $c$; for details see the end of Section \ref{S:proofofa}.

One perspective for studying Hodge ideals, introduced by Saito, is to consider the induced microlocal $V$-filtration on $\so_X$ associated to $f$ (see more details about the $V$-filtration in Section \ref{S:microlocal}). When $D$ is a reduced and integral divisor, Saito showed in \cite{Sai16} that 
$$I_k(D)=\tilde{V}^{k+\alpha}\so_X\mod (f).$$
When we consider an effective $\fQ$-divisor $D=\alpha H$, in \cite[Definition 3.1]{MP18b} the authors define another series of ideal sheaves $\{\tilde{I}_k(D)\}_{k\in\rN}$; see Section \ref{S:microlocal}. In this setting, when $0<\alpha\le 1$, it is easy to see that $\tilde{I}_k(D)=\tilde{V}^{k+\alpha}\so_X$. In the case when $D=\alpha\cdot Z$, where $0<\alpha\le 1$, $Z=\diV(f)$ reduced  and $f\in\so_X(X)$, the authors \cite[Theorem A$'$]{MP18b} obtained a formula for $I_k(D)$ in terms of the $V$-filtration on $\dmM(f^{1-\alpha})$, from which in particular it follows that $I_k(D)=\tilde{I}_k(D)\mod (f)$.  It is natural to wonder to what extent equality holds without modding out by $(f)$. Although $I_0(D)=\tilde{I}_0(D)$ is always true, $I_k(D)$ and $\tilde{I}_k(D)$ are usually not the same for $k\ge 1$ (see, e.g. \cite[Remark (ii) in \S 2.4]{Sai16} for $\alpha=1$; see also \cite[Remark 9.8]{Pop18} for examples of $\fQ$-divisors). However, if $D=\alpha Z$ with $0<\alpha\le1$ and $Z$ is defined by a weighted homogeneous polynomial with an isolated singularity at the origin, we can give a criterion for equality between $I_k(D)$ and $\tilde{I}_k(D)$:

\begin{proposition}\label{P:hodgemicroequality} 
If $D=\alpha Z$, where $0<\alpha\le 1$ and $Z=(f=0)$ is an integral and reduced effective divisor defined by $f$, a weighted homogeneous polynomial with an isolated singularity at the origin, then for any $k\in\rN$, if $I_k(D)=\tilde{I}_k(D)=(\row xn)^m$ for some $m\in\rN$, then $I_{k+1}(D)= \tilde{I}_{k+1}(D)$. \end{proposition}
When $k=0$, we prove that the converse statement is true, i.e., if $I_0(D)=\tilde{I}_0(D)\neq (\row xn)^m $ for any $m\in \rN$, then $I_1(D)\neq \tilde{I}_1(D)$. See details in Proposition \ref{P:firstofconj}. So it seems plausible to make the following:

\begin{conj} \label{Cj:hodgemicro}
If $D=\alpha Z$, where $0<\alpha\le 1$ and $Z=(f=0)$ is an integral and reduced effective divisor defined by $f$, a weighted homogeneous polynomial with an isolated singularity at the origin, then for any $k\in\rN$, $I_{k+1}(D)= \tilde{I}_{k+1}(D)$ if and only if $I_k(D)=\tilde{I}_k(D)=(\row xn)^m$ for some $m\in \rN$. 
\end{conj}
\endgroup

Finally, in Section \ref{S:nondegenerate} we extend Theorem \ref{thma} to the case that $D=\alpha Z$, where $0<\alpha\le 1$ and $Z=(f=0)$ is an integral and reduced effective divisor defined by $f$, a germ of holomorphic function that is convenient and has non-degenerate Newton boundary with respect to the local coordinates.
\section*{Acknowledgement}
I would like to express my sincere gratitude and appreciation to my advisor Mihnea Popa for suggesting the problem, and for the continuous support of this project. Special thanks to Morihiko Saito for patiently answering my questions and explaining to me many ideas of his results. I would like to thank Yajnaseni Dutta, Mircea Musta\c{t}\u{a}, Sebasti\'{a}n Olano, Claude Sabbah, Lei Wu, Stephen Shing-Toung Yau and Huai-Qing Zuo for answering my questions, suggesting useful references and very helpful discussions.
\section{Preliminaries}
\setcounter{theorem}{\thetmp}
\subsection{Filtered $\doD$-modules and Hodge ideals associated with $\mathbb{Q}$-divisors}\label{S:filtereddm}

Let $X$ be a smooth complex variety, and $D$ be an effective $\fQ$-divisor on $X$, locally $D=\alpha H$ with support $Z$ and $H=\diV(h)$ for some nonzero $h\in\so_X(X)$ and $\alpha\in \fQ_{>0}$. We denote $\beta=1-\alpha$.

In this setting one associates $D$ the left $\doD_X$-module $\dmM(h^{\beta}):=\so_X(*H)h^{\beta}$, a rank 1 free $\so_X(*H)$-module with an generator the symbol $h^{\beta}$. The left $\doD_X$-module structure is given via the rule
$$D(wh^{\beta}):=\left(D(w)+w\frac{\beta\cdot D(h)}{h}\right)h^\beta, ~\textrm{for any}~D\in \textrm{Der}_{\fC}\so_X.$$
A good filtration is defined on the $\doD_X$-module $\dmM(h^\beta)$, written as
$(\dmM(h^\beta), F_{\bullet}),$
and is called the Hodge filtration of $\dmM(h^\beta)$ (see \cite[Definition 2.10]{MP18a}).
 \begin{definition}\cite[definition after Proposition 4.1]{MP18a}\label{defHI}
Let $D=\alpha H$ with support $Z$, where $H=\diV(h)$ for some nonzero $h\in\so_X(X)$ and $\alpha\in \fQ_{>0}$. We denote $\beta=1-\alpha$. For any $k\in\rN$, we define $k$th Hodge ideal of $D$, denoted as $I_k(D)$, by
$$F_k\dmM(h^{\beta})=I_k(D)\otimes_{\so_X} \so_X(kZ+H)h^{\beta}.$$ 
\end{definition}

To study $F_\bullet \dmM(h^{\beta})$, it is necessary to consider the following construction. Let
$$\iota: X\hookrightarrow X\times\fC,~x\mapsto \big{(}x, h(x)\big{)}$$
be the closed embedding given by the graph of $h$. For any $\doD_X$-module $\dmM$, we consider the $\doD$-module theoretic direct image
$$\iota_{+}\dmM:=\dmM\otimes_{\fC}\fC[\partial_t]$$
 (see for instance \cite[Example 1.3.5]{HTT08}). With this description, multiplication by $t$ is given by
\begin{equation}\label{E:taction}
t(m\otimes\partial_t^j)=hm\otimes\partial_t^j-jm\otimes\partial_t^{j-1}
\end{equation}
and the action of a derivation $D\in \textrm{Der}_{\mathbb{C}}\mathcal{O}_X$ is given by
\begin{equation}\label{E:daction}
D(m\otimes\partial_t^j)=D(m)\otimes\partial_t^j-D(h)m\otimes\partial_t^{j+1}.
\end{equation}
The Hodge filtration on $\iota_+ \dmM$ is given by: 
$$F_k\iota_+\dmM=\bigoplus_{j=0}^k F_{k-j}\dmM\otimes \partial_t^j$$
 (see \cite[Lemma 3.2.4]{Sai88}).  

In particular, for the special $\doD_X$-module $\dmM(h^\beta)$ ($H$ may not be reduced) mentioned in the beginning, every element in $\iota_{+}\dmM(h^\beta)$ can be written uniquely as a finite sum
$$\sum_{j\ge 0}g_jh^{\beta}\otimes\partial_t^j,~\textrm{with}~g_j\in \mathcal{O}_X(*Z).$$
The Hodge filtration is written as
\begin{equation}\label{E:filtration}
F_k\iota_+\dmM(h^\beta)=\bigoplus_{j=0}^k F_{k-j}\dmM (h^\beta)\otimes \partial_t^j.
\end{equation}
The following lemma is probably well known to the experts, yet we include a proof here for the benefit of the reader. 
\begin{lemma}\label{L:jjproperty}
Let $D=\alpha H$ and $H=\diV(h)$ for some nonzero $h\in\so_X(X)$ and $1-\beta=\alpha\in \fQ_{>0}$.
 Let $j: X\times(\fC-\{0\})\to X\times\fC$ be the natural inclusion. Then:
$$j_*j^*F_k\iota_+\dmM(h^\beta)=\sum_{i=0}^k\mathcal{M}(h^\beta)\otimes\partial_t^i.$$
\end{lemma}
\begin{proof}Let $U=X\times(\fC-\{0\})$. Since $j:U\to X\times \fC$ is the natural inclusion, we know
$$j_*j^*F_k\iota_+\dmM(h^\beta)=\{x\in\iota_+\dmM(h^\beta):x|_U\in F_k\iota_+\dmM(h^\beta)\}.$$
This means, for any $x\in j_*j^*F_k\iota_+\dmM(h^\beta)$, locally, there exists $p\ge 0$ such that $t^p(x)\in F_k\iota_+\dmM(h^\beta)$.
So by equation (\ref{E:filtration}) we have
 $$t^p(x)=\sum_{j=0}^k x_j\otimes \partial_t^j,~\textrm{where}~x_j\in F_{k-j}\dmM(h^\beta).$$
 Then by equation (\ref{E:taction}) we can easily get the inverse formula for the $t$ action, so we obtain 
 $$x\in\sum_{i=0}^k\dmM(h^\beta)\otimes\partial_t^i.$$
Therefore 
$$j_*j^*F_k\iota_+\dmM(h^\beta)\subseteq\sum_{i=0}^k\mathcal{M}(h^\beta)\otimes\partial_t^i.$$

Conversely, $$\left\{\frac{1}{h^m}h^\beta\otimes\partial_t^s\right\}_{m\ge 0,~0\le s\le k}$$ generates $\sum\limits_{i=0}^k\mathcal{M}(h^\beta)\otimes\partial_t^i$. By equation (\ref{E:taction}) we have, 
$$
t^{m+s}\left(\frac{1}{h^{m}}h^\beta\otimes\partial_t^s\right)=\sum_{j=0}^s\frac{h^{k+1}}{h^{k-j+1}}{m+s\choose s-j}\frac{s!}{j!}h^\beta\otimes\partial_t^j.\\
$$
Now, using the definition of the multiplier ideal $\mathscr{I}(B)$ of a $\fQ$-divisor $B$ (see \cite[Definition 9.2.1]{Laz04}), we have:
  $$\so_X(-H)=g_*\so_Y(-g^*H)\subset g_*\so_Y(K_{Y/X}-[(1-\epsilon)g^*H])=\mathscr{I}((1-\epsilon)H)=I_0(H),$$
  hence $(h)\subseteq I_0(H)$, which leads to $h^{k+1}\in I_{k-j}(H)$ for $0\le j\le k$ since 
  $$F_0\so_X(*H)\subset \cdots\subset F_{k}\so_X(*H)\subset F_{k+1}\so_X(*H)\cdots$$
  is an increasing filtration. Thus for any $m\ge 0$ and $0\le s\le k$, we obtain $t^{m+s}(\frac{1}{h^{m}}h^\beta\otimes\partial_t^s)\in F_p\iota_+\dmM(h^\beta)$. Therefore 
  $$j_*j^*F_k\iota_+\dmM(h^\beta)\supseteq\sum_{i=0}^k\mathcal{M}(h^\beta)\otimes\partial_t^i,$$
  and we finish the proof.
\end{proof}
\subsection{The rational $V$-filtrations on $\big{(}\so_X(*H),F\big{)}$ and $\big{(}\dmM(h^{1-\alpha}), F\big{)}$.}\label{S:Vfiltration}
For any complex manifold $Y$, let $X$ be a smooth hypersurface on $Y$, with a local coordinate $t$.

Malgrange \cite{Mal83} proved that $\iota_+\so_X$ admits an integral $V$-filtration along $X\times\{0\}$. Kashiwara \cite{Kas83} extended the result to the case $\iota_+\dmM$, for $\dmM$ regular holonomic, while Saito \cite[Definition 3.1.1]{Sai88} considered the more general rational $V$-filtration defined as follows:
 \begin{definition}\label{Def:rVfil}
Let $M$ be a coherent $\doD_{Y}$-module. A decreasing filtration ($V^{\gamma} M)_{\gamma\in\fQ}$ on $M$, indexed by $\fQ$, is called a rational $V$-filtration along $X$ if the following conditions are satisfied:
\begin{enumerate}
\item Each $V^{\gamma}M$ is a coherent module over $\doD_X[t,\partial_tt]$ and $\bigcup\limits_{\gamma\in\fQ}V^\gamma M=M$.

\item $t\cdot V^{\gamma}M\subseteq V^{\gamma+1}M$, for any $\gamma\in\fQ$ (with equality if $\gamma>0$).

\item $\partial_t\cdot V^{\gamma}M\subseteq V^{\gamma-1}M$, for any $\gamma\in\fQ$.
\item For every $\gamma\in\fQ$, if we put $V^{>\gamma}M=\cup_{\gamma'>\gamma}V^{\gamma'}M$, then $\partial_tt+\gamma$ is nilpotent on 
$$\textrm{Gr}_V^\gamma M:=V^{\gamma}M/V^{>\gamma}M.$$
\end{enumerate}
\end{definition}
Saito \cite[Definition 3.2.1]{Sai88} further consider the rational $V$-filtration compatible with the increasing filtration $F$ defined as follows:
\begin{definition}
Let $(M,F)$ be a filtered coherent $\doD_Y$-module, i.e.~$\textrm{Gr}^F_\bullet$ are coherent $\textrm{Gr}^F_\bullet\doD_Y$ module. We say that $(M,F)$ admits a rational $V$-filtration along $X$, if:
\begin{itemize}
\item $M$ admits a rational $V$-filtration along $X$ as defined in Definition \ref{Def:rVfil}.

\item $t\cdot F_pV^\alpha M=F_pV^{\alpha+1}M$ for any $\alpha>0$.

\item $\partial_t\cdot F_p\textrm{Gr}_V^\alpha M=F_{p+1}\textrm{Gr}_V^{\alpha-1}$ for any $\alpha<1$.
\end{itemize}
\end{definition}

The rational $V$-filtration on $(M,F)$ defined as above exists in the case when $(M,F)$ underlies a mixed Hodge module. It is well known that when the $V$-filtration exists, it is unique. In the paper, we will mainly be concerned with the rational $V$-filtration on filtered $\doD_{X\times\fC}$-modules of the form $(\iota_+\so_X, F_\bullet)$ or $\big{(}\iota_+\dmM(h^\beta), F_\bullet\big{)}$. 
 
 First, we recall a useful result to compute the Hodge filtration on a $\doD_{X\times\fC}$-module from its restriction on $U=X\times (\fC-\{0\})$ as follows:
\begin{lemma} \cite[Remark 3.4]{MP18b}\label{L:regular quasiunipotent property} Let $D=\alpha H$ and $H=\diV(h)$ for some nonzero $h\in\so_X(X)$ and $1-\beta=\alpha\in \fQ_{>0}$.
 Let $j: X\times(\fC-\{0\})\to X\times\fC$ be the natural inclusion. We have
$$F_k\iota_+\dmM(h^\beta)=\sum_{i\ge 0}\partial_t^i\big{(}V^0\iota_+\dmM(h^\beta)\cap j_*j^*F_{k-i}\iota_+\dmM(h^\beta)\big{)}\hspace{10pt}\textrm{for all}~k,$$
where $j: U\hookrightarrow X\times\fC$ is the inclusion.
\end{lemma}
The above lemma is a direct application of a more general result proved by M. Saito (see \cite[Remark 3.2.3]{Sai88}), which applies since $\big{(}\iota_+\dmM(h^\beta),F_\bullet\big{)}$ is a filtered direct summand of a mixed Hodge module. Combining this with Lemma \ref{L:jjproperty}, we get:
\begin{equation}\label{firststep}
F_k\iota_+\dmM(h^\beta)=\sum_{i\ge 0}\partial_t^i\left(V^0\iota_+\dmM(h^\beta)\cap \sum_{i=0}^k\dmM(h^\beta)\otimes \partial_t^i\right).
\end{equation}

The following result is also related to the fact that $\iota_+\dmM(h^\beta)$ is a filtered direct summand of a mixed Hodge module. 
\begin{lemma} \cite[Lemma 4.5]{MP18b} \label{L:tbijection}
For every $\gamma\in\fQ$, we have
$$t\cdot V^{\gamma}\iota_+\dmM(h^\beta)=V^{\gamma+1}\iota_+\dmM(h^\beta).$$
Moreover, for every $k\in\mathbb{Z}$ and every $\gamma\ge 0$, we have
$$t\cdot V^{\gamma}F_k\iota_+\mathcal{M}(h^\beta)=V^{\gamma+1}F_k\iota_+\mathcal{M}(h^{\beta}).$$
\end{lemma}

We will also make use of the following technical result:
\begin{lemma} \cite[Lemma 3.1.7]{Sai88} \label{L:localizationHM}
Let $u: M\to N$ be a morphism of coherent $\doD_X$-modules equipped with filtration $V$. Let $X^*:=X\times(\fC-\{0\})$. If $u|_{X^*}: M|_{X^*}\to N|_{X^*}$ is an isomorphism, then $u: V^\alpha M\cong V^\alpha N$ for all $\alpha>0$.
\end{lemma}

When we work with a $\fQ$-divisor $D$, locally $D=\alpha\cdot H$ where $H=\diV(h)$ and $0<\alpha\le1$, a key point is to express the $V$-filtration on $\iota_+\dmM(h^\beta)$ in terms of the more common one on $\iota_+\so_X(*H)$. This was done in the following:
\begin{proposition} \cite[Proposition 2.6 and Proposition 2.8]{MP18b}
There is an $\so_X(*H)$-linear isomorphism $\Phi: \iota_{+}\dmM(h^\beta)\to \iota_+\so_X(*H)$ such that
$$\Phi\left(\sum_{i=0}^k\phi_ih^{\beta}\otimes \partial_t^i\delta\right)=\sum_{i=0}^kg_i\otimes\partial_t^i\delta,$$
where
$$\frac{\phi_i}{h^i}=\sum_{j=i}^k{j\choose i}Q_{j-i}(-\beta)\frac{g_j}{h^j},$$
with 
\begin{equation}\label{E:qk}
Q_k(\alpha)=\begin{cases}
\alpha(\alpha+1)\cdots(\alpha+k-1),~\textrm{if}~k\ge 1;\\
1,~\textrm{if}~k=0.\\
\end{cases}
\end{equation}
In particular, we have
$$\Phi\big{(}V^\gamma\iota_+\dmM(h^\beta)\big{)}=V^{\gamma-\beta}\iota_+\so_X(*H)\hspace{10pt}\textrm{for every}~\gamma\in\fQ.$$ \label{Pr: isobeta}
\end{proposition}
The above proposition will lead to an bijection between some set of sections of certain type appearing in the equation (\ref{firststep}), 
$$\sum_{i=0}^k\dmM(h^\beta)\otimes\partial_t^i~ \textrm{on the}~\doD_{X\times\fC}\textrm{-module}~ \iota_+\dmM(h^\beta)$$
 and a similar set of sections, 
 $$\sum_{i=0}^k\so_X(*H)\otimes \partial_t^i~ \textrm{on the} ~\doD_{X\times\fC}\textrm{-module}~ \iota_+\so_X(*H).$$ 
 Moreover, we will use the last statement for $\gamma=1$, i.e.
$$\Phi\big{(}V^1\iota_+\dmM(h^\beta)\big{)}=V^{\alpha}\iota_+\so_X(*H).$$
Also, Lemma \ref{L:localizationHM} implies that
$$V^\alpha\iota_+\so_X(*H)=V^\alpha\iota_+\so_X,~\forall ~\alpha>0.$$
Using the above two equations, we are able to improve the equation (\ref{firststep}) and give a better description of the Hodge filtrations on $\iota_+\dmM(h^\beta)$. The details are in the proof of Proposition \ref{prop31}.

\subsection{Weighted homogeneous polynomial with isolated singularities.}\label{S:quasih}
We choose and fix a local coordinate system $(\row xn)$ around an isolated singularity. We say that $f$ is a weighted homogeneous polynomial of weights $w=(\row wn)$ if $f$ is a linear combination of finitely many monomials $x^{\nu}$ such that $$\sum\limits_{i=1}^nw_i\nu_i=1,~\textrm{where}~w_i\in\fQ_+.$$

We denote $\partial_if:=\frac{\partial f}{\partial x_i}$ for any $i$, $1\le i\le n$. Since $f$ has an isolated singularity at the origin, it is well known that the Jacobian algebra
$$\mathcal{A}_f:=\fC\{\row xn\}/(\partial_1f,\ldots, \partial_nf)$$
is a finite dimensional $\fC$-vector space. Denote a monomial basis for the Jacobian algebra $\mathcal{A}_f$ as $\{\row v\mu\}$, where $\mu$ is the dimension of $\mathcal{A}_f$ (usually called the Milnor number of $f$).

\begin{remark}\label{weightsandlct}
When $f$ is a weighted homogeneous polynomial with an isolated singularity at the origin, we know more about $\Vw(v_i)$ for $1\le i\le \nu$. It is well known (see e.g. \cite[remark after 2.8]{Sai94}) that $$\max_i\Vw(v_i)=n-\sum_{j=1}^nw_j=n-\tilde{\alpha}_f.$$
\end{remark}

\section{Proof of Theorem A}\label{S:proofofa}
For the moment we continue to consider an arbitrary $\fQ$-divisor $D$ such that $D=\alpha\cdot H$ where $\alpha>0$ and $H=\textrm{div}(h)$.
We define the natural projection $\varphi: \iota_+\dmM(h^\beta)\to \dmM(h^\beta)$ by:
$$\varphi\left(\sum_{i=0}^km_i\otimes\partial_t^i\right)=m_0,\hspace{5pt}\forall~m_i\in \dmM(h^\beta),~\forall ~k\in\rN.$$
\begin{proposition}\label{prop31}
For any effective $\fQ$-divisor $D$ as above, the Hodge filtration on $\dmM(h^\beta)$ satisfies:
$$F_p\dmM(h^\beta)=\varphi\circ t^{-1}\circ\Phi^{-1}\left(V^\alpha\iota_+\so_X\cap\big{(}\sum_{i=0}^p\so_X\otimes\partial_t^i\big{)}\right),$$
where $\varphi$ is defined as above and $\Phi$ is the isomorphism in Proposition \ref{Pr: isobeta}.
\end{proposition}

\begin{proof}
Applying the map $\varphi$ to both sides of the equation (\ref{firststep}), after simplifying the sum, we get
$$F_p\dmM(h^\beta)=\varphi\left(V^0\iota_+\dmM(h^\beta)\cap \big{(} \sum_{i=0}^p\mathcal{M}(h^\beta)\otimes\partial_t^i\big{)}\right).$$
By Proposition \ref{Pr: isobeta}, we obtain
$$\Phi^{-1}(a\otimes \partial_t^k)=\sum_{i=0}^k{k\choose i}Q_{k-i}(-\beta)\frac{a}{h^{k-i}}h^\beta\otimes\partial_t^i,$$
and
$$\Phi(bh^\beta\otimes \partial_t^k)=\sum_{i=0}^k{k\choose i}Q_{k-i}(\beta)\frac{b}{h^{k-i}}\otimes\partial_t^i.$$
Since $\Phi$ is bijective, these two formulas imply that  $$\Phi^{-1}\left(\sum_{i=0}^p\so_X(*H)\otimes\partial_t^i\right)=\sum_{i=0}^p\dmM(h^\beta)\otimes\partial_t^i.$$

 Letting $\gamma=1$ in Proposition \ref{Pr: isobeta}, we get
 $$V^1\iota_+\dmM(h^\beta)=\Phi^{-1}\big{(}V^\alpha\iota_+\so_X(*H)\big{)}.$$
Moreover, $$t: V^0\iota_+\dmM(h^\beta)\to V^1\iota_+\dmM(h^\beta)$$ is bijective (by Lemma \ref{L:tbijection} with $\gamma=0$) and $$t:\sum_{i=0}^p\dmM(h^\beta)\otimes\partial_t^i\to\sum_{i=0}^p\dmM(h^\beta)\otimes\partial_t^i$$ is bijective for any $p\in\rN$ (in fact we can write down the formula for $t^{-1}$).
Thus we obtain:
$$F_p\dmM(h^\beta)=\varphi\circ t^{-1}\circ \Phi^{-1}\left(V^\alpha\iota_+\mathcal{O}_X(*H)\cap \sum_{i=0}^p\so_X(*H)\otimes\partial_t^i\right).
$$
Notice that $V^\alpha\iota_+\so_X(*H)=V^\alpha\iota_+\so_X$ for every $\alpha>0$ by Lemma \ref{L:localizationHM}, so we can replace $\so_X(*H)$ by $\so_X$ in the right hand side of the equation to get the desired result.
\end{proof}
From now on, we denote $\Psi:=\varphi\circ t^{-1}\circ\Phi^{-1}$. This is an $\so_X$-module homomorphism having the following property:
\begin{lemma}\label{L:psidxaction}
$$
\Psi\big{(}P(a\otimes \partial_t^k)\big{)}=Q_k(\alpha)P\cdot\left(\frac{a}{h^{k+1}}h^{1-\alpha}\right),\hspace{10pt}~\textrm{for any}~P\in\mathscr{D}_X,~\forall~a\in\so_X,
$$
where $Q_k(\alpha)$ is the same as in (\ref{E:qk}).
\end{lemma}
\begin{proof}
When $P\in F_0\doD_X$, it is easy to check the desired equation is true.
 
 Inductively, for some $l\in\rN$, we assume for any $\nu=(\row \nu n)$ with $\abs{\nu}=l$, the following equation is true:
 $$\Psi\big{(}\partial_x^\nu(a\otimes\partial_t^k)\big{)}=Q_k(\alpha)\partial_x^\nu\cdot\left(\frac{a}{h^{k+1}}h^{1-\alpha}\right),~\forall ~a\in\so_X, ~\forall ~k\ge 0.$$
 Then we get for any $\nu$ with $\abs{\nu}=l$, 
  \begin{align*}
 (\partial_x^\nu\partial_{x_i})\cdot\left(\frac{a}{h^{k+1}}h^{1-\alpha}\right)&=\partial_x^\nu\cdot\left(\frac{\partial_{x_i}a}{h^{k+1}}h^{1-\alpha}-\frac{(\alpha+k)\partial_{x_i}fa}{h^{k+2}}h^{1-\alpha}\right)\\
 &=\Psi\left(\frac{1}{Q_k(\alpha)}\partial_x^{\nu}\big{(}\partial_{x_i}a\otimes\partial_t^k\big{)}-\frac{\alpha+k}{Q_{k+1}(\alpha)}\partial_x^{\nu}\big{(}\partial_{x_i}ha\otimes\partial_t^{k+1}\big{)}\right)\\
 &=\frac{1}{Q_k(\alpha)}\Psi\left(\partial_x^\nu\partial_{x_i}(a\otimes\partial_t^k)\right).\\
 \end{align*}
 So by induction, we finish the proof.
\end{proof}
From now on, we focus on the case when $H=Z=(f=0)$ reduced, with $f$ a weighted homogeneous polynomial with an isolated singularity at the origin. In this case we can obtain a formula for the Hodge filtration. The key ingredient for the proof is the following result of Morihiko Saito. The lemma below gives an explicit formula for the Hodge filtration on the $V$-filtration of the algebraic microlocalization of $\iota_+\so_X$. Here the so-called microlocal $V$-filtration on the algebraic microlocalization of $\iota_+\so_X$, i.e. $\so_X[\partial_t,\partial_t^{-1}]:=\so_X\otimes_{\fC}\fC[\partial_t,\partial_t^{-1}]$, is defined as follows (see \cite[(2.1.3)]{Sai94}):
\begin{equation}\label{E:microlocalV}
V^\alpha\so_X[\partial_t,\partial_t^{-1}]=\begin{cases}
V^\alpha\iota_+\so_X\oplus(\so_X\otimes_{\fC}\fC[\partial_t^{-1}]\partial_t^{-1})&\textrm{if}~\alpha\le 1;\\
\partial_t^{-j}V^{\alpha-j}\so_X[\partial_t,\partial_t^{-1}]&\textrm{if}~\alpha>1,~\alpha-j\in (0,1].
\end{cases}
\end{equation}
\begin{lemma}\cite[(4.2.1)]{Sai09} \label{L:421} Assume $H=Z=(f=0)$ reduced, with $f$ a weighted homogeneous polynomial with an isolated singularity at the origin. The graph embedding $\iota: X\to X\times \fC$ is defined by $\iota(x)=\big{(}x, f(x)\big{)}$. Then we have
$$F_kV^{\alpha}\so_X[\partial_t,\partial_t^{-1}]=\sum_{i\le k}F_{k-i}D_{X}(\so^{\ge \alpha+i}\otimes\partial_t^i).$$
\end{lemma}
It turns out that $V^\alpha \iota_+\so_X$, the $\doD_{X\times\fC}$-module involved in the expression of $F_\bullet \dmM(h^\beta)$ in Proposition \ref{prop31}
, can be related to the microlocal $V$-filtration; this leads to:
\begin{proof}[\textbf{Proof of Theorem} \ref{thma}]
By Lemma \ref{L:421}, letting $i: \so_X[\partial_t]\to \so_X[\partial_t, \partial_t^{-1}]$ be the natural inclusion, we have by definition that $i^{-1}(V^\alpha \so_X[\partial_t, \partial_t^{-1}])=V^\alpha\so_X[\partial_t]$, so we obtain
$$F_pV^\alpha\iota_+\so_X= F_p V^\alpha \so_X[\partial_t, \partial_t^{-1}]\cap \so_X[\partial_t]$$
 Using the fact that $\so^{\ge \alpha+i}=\so_X$ for $i\le-1$, and $(\partial_{x_1}f, \cdots,\partial_{x_n}f)\subset \so^{\ge 1}\subset \so^{\ge \alpha}$ since $0<\alpha\le 1$, we further get
$$F_pV^\alpha\iota_+\so_X=\sum_{i=0}^p F_{p-i}D_X(\so^{\ge \alpha+i}\otimes \partial_t^i).$$
Hence by Proposition \ref{prop31}, we obtain
$$F_p\dmM(f^\beta)=\Psi\left(\sum_{i=0}^pF_{p-i}D_X(\so^{\ge \alpha+i}\otimes\partial_t^i)\right),$$
where $\Psi:=\varphi\circ t^{-1}\circ\Phi^{-1}$ is defined as above. Then Lemma \ref{L:psidxaction} implies
$$\Psi\big{(}F_{p-i}D_X(\so^{\ge \alpha+i}\otimes\partial_t^i)\big{)}=F_{p-i}D_X\cdot\left( \frac{\so^{\ge\alpha+i}}{f^{i+1}}f^{1-\alpha} \right)~\hspace{10pt}~\textrm{for any}~0\le i\le p,$$
which completes the proof.
\end{proof}

\begin{proof}[\textbf{Proof of Corollary} \ref{C:B}]
We first show the following statement (\ref{E:1step}) by induction: for $k\in\rN$,
\begin{equation}\label{E:1step}
I_{k+1}(D)=\so^{\ge \alpha+k+1}+\sum_{1\le i\le n,a\in I_{k}(D)}\mathcal{O}_X\big{(}f\partial_ia-(\alpha+k)a\partial_if\big{)}
\end{equation}

By Theorem \ref{thma}, for $p=0$ and $p=1$ we obtain
$$I_0(D)=\so^{\ge \alpha}$$ 
and 
$$I_1(D)f^{1-\alpha}=f^2\cdot F_1\mathcal{O}_X(*Z)f^{1-\alpha}=\so^{\ge \alpha+1}+(f)I_0(D)+\sum_{1\le i\le n, a\in I_0(D)}\so_X\big{(}f\partial_ia-\alpha a\partial_if\big{)}.$$
So we get $(f)I_0(D)\subset \so^{\ge \alpha+1}$ and thus the statement (\ref{E:1step}) is true for $k=0$.

Now assume the statement is true for $k=l$. Then
\begin{align*}
I_{l+1}(D)&=\so^{\ge \alpha+l+1}+(f)I_l(D)+\sum_{1\le i\le n, a\in I_l(D)}\so_X\big{(}f\partial_ia-(\alpha+l)a\partial_if\big{)}\\
&=\so^{\alpha+l+1}+(f)\so^{\ge \alpha+l}+(f)\sum_{1\le j\le n, b\in I_{l-1}(D)}\so_X\big{(}f\partial_jb-(\alpha+l-1)b\partial_jf\big{)}\\
&\quad +\sum_{1\le i\le n, a\in I_l(D)}\so_X\big{(}f\partial_ia-(\alpha+l)a\partial_if\big{)}\\
\end{align*}
Since $(f)\so^{\ge\alpha+l}\subset \so^{\ge \alpha+l+1}$ and $$f\partial_j(fb)-(\alpha+l)(fb)\partial_jf=f\big{(}f\partial_jb-(\alpha+l-1)b\partial_jf\big{)},$$ we conclude that the statement is true when $k=l+1$. So by induction the statement (\ref{E:1step}) is true.

Then it suffices to show:
 $$\so^{\ge \alpha+k+1}\cap(\partial_1f,\ldots, \partial_nf)\subseteq \sum_{1\le i\le n,a\in I_{k}(D)}\mathcal{O}_X\big{(}f\partial_ia-(\alpha+k)a\partial_if\big{)}.$$
 Take any monomial $b\in \so^{\ge \alpha+k+1}\cap (\partial_1f,\ldots, \partial_nf)$, write $b=B_1\partial_1f+B_2\partial_2f+\cdots+B_n\partial_nf$.  
 We claim that:
  \begin{equation}\label{E:claim}
  B_i\partial_if\in \sum_{1\le i\le n,a\in I_{k}(D)}\mathcal{O}_X\big{(}f\partial_ia-(\alpha+k)a\partial_if\big{)}, \forall ~1\le i\le n.
  \end{equation}
To prove the claim, we just show the case for $i=1$, since for $2\le i\le n$ the argument is completely similar. If $B_1=0$, then $\partial_1fB_1=0$. Now suppose $0\neq B_1=\sum_{A\in J_1}b_A x^{A}$. Take any $A\in J_1$. If $\partial_1(x^A)=0$, then $\partial_1fx^{A}$ belongs to the right hand side of (\ref{E:claim}); otherwise we have $\partial_1(x^A)\neq 0$, so we get $\Vw(\partial_1(x^A))\ge \alpha+k$ and hence $\partial_1(x^A)\in I_k(D)$. We can easily check that
$$(\alpha+k)x^A\partial_1f=(a_1-1)\big{(}f\partial_1(x^A)-(\alpha+k)x^A\partial_1f\big{)}-x_1\big{(}f\partial_1(\partial_1x^A)-(\alpha+k)\partial_1f\cdot \partial_1x^A{)}.$$
 Thus $$B_1\partial_1f=\sum_{A\in J_1}b_A(x^A\partial_1f)\in\sum_{1\le i\le n,a\in I_{k}(D)}\mathcal{O}_X\big{(}f\partial_ia-(\alpha+k)a\partial_if\big{)}.$$ 
 Similarly, we can show $B_i\partial_if$ belongs to the right hand side of (\ref{E:claim}) for $2\le i\le n$, and hence the claim is true. Therefore 
 $$b\in\sum_{1\le i\le n,a\in I_{k}(D)}\mathcal{O}_X\big{(}f\partial_ia-(\alpha+k)a\partial_if\big{)}.$$
\end{proof}

As an application of Theorem \ref{thma}, we prove the following property of Hodge ideals:
\begin{corollary}\label{C:hodgeidealseq}
If $D=\alpha Z$, where $0<\alpha\le 1$ and $Z=(f=0)$ is an integral and reduced effective divisor defined by $f$, a weighted homogeneous polynomial with an isolated singularity at the origin, then we have:
$$I_{k+1}(D)\subseteq I_k(D)~\hspace{10pt}\textrm{for any}~k\in\mathbb{N}.$$
\end{corollary}
\begin{proof}
When $k=0$, by Corollary \ref{C:B}, we easily get $$I_1(D)\subseteq I_0(D).$$
(In fact, even without the assumption for $f$, we always have $I_1(D)\subseteq I_0(D)$ when $D=\alpha Z$ with $\alpha\le 1$ by \cite[Corollary 5.5]{MP18b}.)

Now we assume $I_{l}(D)\subseteq I_{l-1}(D)$ for some $l\in\rN$. Then, for $1\le i\le n$ and for any $a\in I_{l}(D)\subseteq I_{l-1}(D)$, by Corollary \ref{C:B}, we have $f\partial_ia-(\alpha+l-1)a\partial_if\in I_l(D)$. By adding $-a\partial_if\in I_l(D)$, we know that $f\partial_ia-(\alpha+l)a\partial_if \in I_l(D)$. Then again by Corollary \ref{C:B}, we get $I_{l+1}(D)\subseteq I_l(D)$. So by induction the proof is complete.
\end{proof}
\begin{proof}[\textbf{Proof of Corollary} \ref{C:generatinglevel}]
As mentioned in Remark \ref{weightsandlct}, for any weighted homogeneous function $f$ with an isolated singularity at the origin, $$\max_{j}\{\Vw(v_j)\}=n-\tilde{\alpha}_f,$$
 where $\{v_1,\cdots, v_{\mu}\}$ is a fixed monomial basis of Jacobian algebra $\fC\{\row xn\}/(\partial f)$. 
 
Let $$J_q(D)f^{1-\alpha}:=F_1\mathscr{D}_X\cdot\left(\frac{I_{q-1}(D)}{f^{q}}f^{1-\alpha}\right)f^{q+1}.$$
By Corollary \ref{C:B}, it is enough to show that 
$$\sum_{v_j\in A^{\ge \alpha+p}}\so_Xv_j\subset J_p(D)~\iff~p\ge [n-\tilde{\alpha}_f-\alpha]+1.$$
First, taking any $q\le [n-\tilde{\alpha}_f-\alpha]$, we have $q+\alpha\le n-\tilde{\alpha}_f$. There exists $v_j$ such that $\Vw(v_j)=n-\tilde{\alpha}_f$. Therefore $v_j \in \so^{\ge q+\alpha}$ and $v_j\notin (f_1,\cdots, f_n)$.

Since $$f=\sum_{i=1}^nw_ix_i\partial_if, $$
we obtain
 $$J_q(D)=(f)I_{q-1}(D)+\sum_{1\le i\le n,~a\in I_{q-1}(D)}\so_X\big{(}f\partial_ia -(q-1+\alpha)a\partial_if\big{)}\subset (f_1,\cdots, f_n).$$
Thus $$\sum_{v_j\in \so^{\ge \alpha+p}}\so_Xv_j\nsubseteq J_q(D).$$

On the other hand, for any $p\ge [n-\tilde{\alpha}_f-\alpha]+1$, we get $p+\alpha\ge [n-\tilde{\alpha}_f-\alpha]+1+\alpha>n-\tilde{\alpha}_f-\alpha+\alpha=n-\tilde{\alpha}_f$. Since $\max\{\Vw(v_j)\}=n-\tilde{a}_f<\alpha+p$, by definition there are no such $v_j$.
\end{proof}
\begin{example}\label{E:cusp}
Let $f=x^2+y^3$. Then $\mathcal{A}_f=\fC 1\oplus\fC y$ and $\Vw(1)=\frac{5}{6}$ and $\Vw(y)=\frac{7}{6}$. Then:
\begin{itemize}
\item When $0<\alpha\le\frac{1}{6}$, $I_0(D)=\mathbb{C}\{x,y\}$, $I_1(D)=(y)+(\partial f)=(x,y)$ and $I_2(D)=(x^2, xy, y^3)$. The generation level is $[2-\frac{5}{6}-\alpha]=1$. We see that $I_1(D)\neq J_1(D)$ and $I_2(D)=J_2(D)$.

\item When $\frac{1}{6}<\alpha\le\frac{5}{6}$, $I_0(D)=\mathbb{C}\{x,y\}$, $I_1(D)=(x^2,xy^2,y^3)+(\partial f)=(x,y^2)$. The generating level is 0.

\item When $\frac{5}{6}<\alpha\le 1$, $I_0(D)=(x,y)$, $I_1(D)=(x^3,x^2y,xy^3,y^4)+\big{(}(1-2\alpha)x^2+y^3, xy^2, xy, x^2+(1-3\alpha)y^3\big{)}=(x^2,xy,y^3)$. The generating level is 0.
\end{itemize}
By using the birational transformation property, in \cite[Example 10.5]{MP18a}, the authors can compute $I_2(D)$ for $\alpha\ge\frac{5}{6}$. It is not clear however how to apply the method to obtain the calculation above for $\alpha<\frac{5}{6}$. But by using Theorem $\ref{thma}$ and Corollary \ref{C:generatinglevel}, we can compute it for all $0<\alpha\le 1$.
\end{example}
As another application of Corollary \ref{C:B}, we illustrate how the roots of the Bernstein-Sato polynomial associated to $f$ relate to the jumping numbers and coefficients associated to ideals of the form $I_k(cZ)/(\partial f)$ in the case when $f$ is a weighted homogeneous polynomial with an isolated singularity at the origin.

 For a germ of holomorphic function $f\in \so_{X,0}\cong\fC\{\row xn\}$, there exists a unique monic polynomial $b_f(s)$ of smallest degree such that there exists $P\in \doD_X[s]$ and
$$b_f(s)f^s=Pf^{s+1}~\hspace{10pt}\textrm{where}~s~\textrm{is a variable}.$$
We call $b_f(s)$ the \emph{Bernstein-Sato polynomial} associated to $f$.

It is usually hard to calculate the Bernstein-Sato polynomial. However, in many cases, a good formula is known, e.g, $f$ is a monomial or $f=x_1^{a_1}+\cdots+x_n^{a_n}$ for some integers $a_i\ge 1$. Moreover, in the case when $f$ has non-degenerate Newton boundary with respect to the local coordinates or in the case when $f$ is semi-quasihomogeneous, a good algorithm has been found (see details in \cite{BGMM89}).
In the case when $f$ is a weighted homogeneous polynomial with an isolated singularity at the origin, the result is much easier to describe. The proof that $b_f(s)$ is written explicitly as follows can be found for instance in Granger's survey \cite[Theorem 4.8]{Gra10}:
\begin{theorem}\label{T:gra10}
Let $f$ be a weighted homogeneous polynomial with an isolated singularity at the origin and with weight $w=(\row wn)$. Let $M=\{\row v\mu\}$ be a monomial basis of the Jacobian algebra $\fC\{\row xn\}/(\partial f)$ where $(\partial f)$ denotes the ideal $(\partial_{1}f,\cdots, \partial_{n}f)$. Let $$E=\{\rho\in\fQ~|~ \exists~v_i\in M~\textrm{such that}~\Vw(v_i)=\rho\}.$$ Then:
$$b_f(s)=(s+1)\prod_{\rho\in E}(s+\rho).$$
\end{theorem}
In the case when $f$ is a weighted homogeneous polynomial with an isolated singularity at the origin and $Z=(f=0)$ is the reduced integral divisor, by Corollary \ref{C:B}, we obtain that $I_k(\lambda Z)/(\partial f)=\so^{\ge k+\lambda}/(\partial f)$ for all $k\in\rN$. By Remark \ref{weightsandlct}, we obtain further that $I_k(cZ)/(\partial f)=0$ for all $k\ge [n-\tilde{\alpha}_f]+1$. So we could relate the elements in the set $E$ to the jumping numbers and jumping coefficients defined as follows:

For any $k\in\rN$ and $k\le [n-\tilde{\alpha}_f]$, there exists $m\in\rN$ and rational numbers $0<c_{k_1}<c_2<\cdots<c_{k_m}<1$ such that for any $1\le i\le m-1$,
$$
\begin{cases}
I_k(\lambda Z)/(\partial f)=I_k(c_{k_i} Z)/(\partial f),&\textrm{if}~\lambda\in[c_{k_i},c_{k_{i+1}});\\
 I_k(c_{k_{i+1}}Z)/(\partial f)\neq I_k(c_{k_i}Z)/(\partial f).&\\
 \end{cases}
$$
We call $c_{k_i}$ jumping numbers of $I_k(cZ)/(\partial f)$ in (0,1). Note that jumping numbers usually cannot be defined directly for $I_k(cZ)$ when $k\ge 1$. This is because some generators of $I_k(cZ)$ might have coefficients depending on $c$, so we may get incomparable ideals. For example, let $f=x^2+y^5$ and $\frac{9}{10}<\alpha\le 1$. As we computed in Example \ref{Ex:notequal}, we cannot define jumping numbers for $I_1(cZ)$ since we get incomparable ideals when we change $c$ in $(\frac{9}{10},1)$ (see also \cite[discussion after Corollary 5.5]{MP18b}). Let $E_k$ denote the set of all jumping numbers of $I_k(cZ)/(\partial f)$ in (0,1) and call
$$N=\{r\in\rN: I_r(Z)/(\partial f)\neq I_{r+1}(\epsilon Z)/(\partial f),~\textrm{where}~0<\epsilon\ll 1\}$$
the set of jumping coefficients associated to $Z$. With these notations, by Theorem \ref{T:gra10}, the jumping numbers and jumping coefficients of $I_k(cZ)/(\partial f)$, for $0\le k\le [n-\tilde{\alpha}_f]$, determine the roots of Bernstein-Sato polynomial as follows:
\begin{corollary}\label{E:bsfroots}
Let $f$ be a weighted homogeneous polynomial with an isolated singularity at the origin and $Z=(f=0)$. Then
$$b_f(s)=(s+1)\prod_{k=0}^{[n-\tilde{\alpha}_f]}\prod_{\rho\in E_k}(s+\rho+k)\prod_{r\in N}(s+r+1).$$
See also Remark \ref{Re:vbroots}.
\end{corollary}
\section{Relation between Microlocal $V$-filtration and Hodge ideals}\label{S:microlocal}
In this section we assume that D is an effective $\fQ$-divisor and $D=\alpha\cdot H$, where $\alpha\in\fQ_{>0}$ and $H=\diV(h)$ with $h\in\so_X(X)$ and denote $Z=H_{red}=\textrm{Supp}(D)$. We assume $\lceil D\rceil=Z=(f=0)$. Then we define another series of ideals as follows:
\begin{definition}\cite[Definition 3.1]{MP18b} \label{microidealdef}
For each $k\ge 0$, we define an ideal sheaf
$$\tilde{I}_k(D):=\left\{v\in\mathcal{O}_X~|~\exists~v_0, v_1.\cdots, v_k=v\in\mathcal{O}_X~\textrm{such that}~\sum_{i=0}^kv_i\otimes\partial_t^i\in V^{\alpha}\iota_+\mathcal{O}_X\right\}.$$
\end{definition}

Recall the microlocalization of $\iota_+\so_X$, i.e. $\so_X[\partial_t,\partial_t^{-1}]$ has the microlocal $V$-filtration defined in (\ref{E:microlocalV}). Notice that it has an increasing filtration (see \cite[(1.2.3)]{Sai94}) given by
$$F_k\so_X[\partial_t, \partial_t^{-1}]=\sum\limits_{i\le k}\so_X\otimes \partial_t^i,~\forall~k\in\rZ.$$
So we get an induced microlocal $V$-filtration on $\so_X$ (identified with $\so_X\otimes 1$), denoted as $$\tilde{V}^\bullet \so_X:=V^\bullet\textrm{Gr}_0^F\so_X[\partial_t,\partial_t^{-1}].$$
When $0<\alpha\le 1$, by definition we have $\tilde{I}_k(D)=\tilde{V}^{k+\alpha}\mathcal{O}_X$. We know the following relationship with Hodge ideals, proved in \cite[Theorem 1]{Sai16} for $\alpha=1$, and extended to $\fQ$-divisors in \cite[Theorem A$'$]{MP18b}:

\begin{proposition} \label{P:HequalM}
If $D=\alpha Z$ where $Z$ is a reduced, effective divisor on $X$, defined by $f\in\so_X(X)$, then
$$I_k(D)=\tilde{I}_k(D)\mod (f),~\forall~k\in\rN.$$
\end{proposition}
When $f$ is a weighted homogeneous polynomial with an isolated singularity at the origin, M. Saito gives an effective way of describing $\tilde{I}_k(D)$ using the lower level $\tilde{I}_{k-1}(D)$ and a monomial basis $\{v_1,\cdots, v_{\mu}\} $ of the Jacobian algebra $\fC\{\row xn\}/(\partial f)$, where $\mu$ is the Milnor number:

\begin{proposition}\cite[Proposition in (2.2)]{Sai16} \label{P:saitomicrovfil} If $D=\alpha Z$, where $0<\alpha\le 1$ and $Z=(f=0)$ is an integral and reduced effective divisor defined by $f$, a weighted homogeneous polynomial with an isolated singularity at the origin \footnote{Note the assumption for \cite[Proposition in (2.2)]{Sai16}  that ``$f$ contains monomials of type $x_i^{a_i}$ for any $i\in [1,n]$" is not necessary, as confirmed by Saito \cite{Sai18}.}, then:
$$\tilde{I}_0(D)=\so^{\ge \alpha}$$
and 
$$\tilde{I}_p(D)=\sum_{v_j\in \so^{\ge \alpha+p}}\mathcal{O}_Xv_j+(\partial_1f,\cdots, \partial_nf)\tilde{I}_{p-1}(D),~\hspace{10pt}\forall~p\ge 1.$$
\end{proposition}
\begin{remark}\label{Re:vbroots}
One consequence of Saito's result is that the roots of the Bernstein-Sato polynomial associated to $f$ are related to the jumping numbers of $\tilde{V}^c\so_X$. In the setting of Theorem \ref{T:gra10}, by Proposition \ref{P:saitomicrovfil}, the microlocal $V$-filtration $\tilde{V}^\bullet\so_X$ modding out $(\partial f)$ determines the following jumping numbers:
There exists rational numbers $\tilde{\alpha}_f=c_1<c_2<\cdots<c_m=n-\tilde{\alpha}_f$ such that for any $1\le i\le m-1$,
$$
\begin{cases}
\tilde{V}^\lambda\so_X/(\partial f)=\tilde{V}^{c_i}\so_X/(\partial f),&\textrm{if}~\lambda\in[c_i,c_{i+1});\\
 \tilde{V}^{c_{i+1}}\so_X/(\partial f)\neq \tilde{V}^{c_i}\so_X/(\partial f).&\\
\end{cases}
$$
With the above notations and by Theorem \ref{T:gra10}, the Bernstein-Sato polynomial is:
$$b_f(s)=(b+1)\prod_{i=1}^m(s+c_i).$$
Notice that $I_k(cZ)=\tilde{I}_k(cZ)\mod (f)$ by Proposition \ref{P:HequalM}, and we have $f\in (\partial f)$ by the definition of weighted homogeneous polynomials. Thus $I_k(cZ)=\tilde{I}_k(cZ)\mod (\partial f)$. So it is not surprising that the roots of the Bernstein-Sato polynomial are also related to the jumping numbers and jumping coefficients associated to ideals of the form $I_k(cZ)/(\partial f)$ as in  Corollary \ref{E:bsfroots}.
\end{remark}
Now we give the proof of Proposition \ref{P:hodgemicroequality}, which gives a criterion for equality between $I_k(D)$ and $\tilde{I}_k(D)$:
 \begin{proof}[\textbf{Proof of Proposition} \ref{P:hodgemicroequality}]
 Comparing $I_{k+1}(D)$ and $\tilde{I}_{k+1}(D)$ by Corollary \ref{C:B}  and by Proposition \ref{P:saitomicrovfil}, clearly it suffices to show that, if $I_k(D)=\tilde{I}_k(D)=(x_1,\ldots, x_n)^m$ for some $m\in \rN$, then 
 $$\sum_{i=1}^n\sum_{a\in (x_1,\ldots, x_n)^m}\so_X\big{(}\partial_ia f-(\alpha+k)a\partial_if\big{)}=(\partial_if,\ldots,\partial_nf)(x_1,\ldots, x_n)^m.$$
 
Fix $k\in\rN$. Now suppose $ I_k(D)=(\row xn)^m$ for some $m\in \rN$ and take any element $a\in (x_1,\ldots, x_n)^m$. 
Since $f$ is a weighted homogeneous polynomial, we have the formula $$f=\sum_{i=1}^nw_ix_i\partial_if.$$ 
Then we get $$f\partial_ia=\sum_{j=1}^nw_jx_j\partial_jf\partial_ia.$$
Clearly $x_j\partial_ia\in (\row xn)^m$, so $f\partial_ia-\alpha\partial_ifa\in (\partial_1f,\ldots, \partial_nf)(x_1,\ldots, x_n)^m$. Thus
$$\sum_{i=1}^n\sum_{a\in (x_1,\ldots, x_n)^m}\so_X\big{(}\partial_ia f-(\alpha+k)a\partial_if\big{)}\subseteq(\partial_if,\ldots,\partial_nf)(x_1,\ldots, x_n)^m.$$
On the other hand, for any $A\in\rN^n$ with $\abs{A}=m-1$, and for any $i\in\{1,\cdots,n\}$, consider
$$L_i(A):=\partial_i(x^{A}x_i)f-(\alpha+k) \partial_ifx^{A}x_i=\big{(}(a_i+1)w_i-\alpha-k\big{)}x^{A}x_i\partial_if+\sum_{j\neq i}(a_i+1)w_jx^{A}x_j\partial_jf.$$
After easy computation, we can express $x^{A}x_i\partial_if$ as
  $$x^{A}x_i\partial_if=\frac{(a_i+1)\sum_{j=1}^n w_jL_j(A)}{(\alpha+k)\big{(}\Vw(x^{A})-\alpha-k\big{)}}-\frac{L_i(A)}{\alpha+k}.$$
 Moreover, for any $B\in\rN^n$ with $\abs{B}=k$ and $b_i=0$, $x^{B}\partial_if=-\frac{1}{\alpha+k}\big{(}\partial_i(x^{B})f-(\alpha+k)x^{B}\partial_if\big{)}$. Hence, for any $i$, and any $a\in (x_1,\ldots, x_n)^m$, we obtain
 $$a\partial_if\in \sum_{i=1}^n\sum_{a\in I_k(D)}\so_X\big{(}f\partial_ia-(\alpha+k)a\partial_if\big{)}.$$
 Thus we proved the reverse inclusion of sets and finish the proof.
\end{proof}
Moreover we prove in the following that Conjecture \ref{Cj:hodgemicro} is true when $k=0$:
\begin{proposition}\label{P:firstofconj} We assume $D=\alpha Z$, where $0<\alpha\le 1$ and $Z=(f=0)$ is an integral and reduced effective divisor defined by $f$, a weighted homogeneous polynomial with an isolated singularity at the origin. If $I_0(D)\neq(\row xn)^p$ for any $p\in \rN$, then $I_1(D)\neq \tilde{I}_1(D)$. 
\end{proposition}
\begin{proof}
We assume that $I_0(D)$ is not equal to $(\row xn)^p$ for any $p\in\rN$. Notice that $I_0(D)=\so^{\ge \alpha}$ is a monomial ideal. So there exists a nonempty set $E$ of vectors in $\rN^n$ such that for any $\eta=(\eta_1,\cdots, \eta_n)\in E$, there exist two nonempty sets $I_{\eta}$, $J_\eta$, and $I_\eta\cup J_\eta=\{1,\ldots, n\}$, such that

$$x^\eta x_{i}\in I_0(D)=\so^{\ge \alpha},~\forall ~i\in I_\eta$$
and 
$$x^\eta x_{j}\notin I_0(D)=\so^{\ge \alpha},~\forall~ j\in J_\eta.$$
For any $\eta\in E$, we get
$x^\eta\notin \so^{\ge\alpha}$ because $x^\eta x_j\notin \so^{\ge \alpha}$ for any $ j\in J_\eta$. Then we know 
\begin{equation}\label{eq2}
x^\eta x_t\partial_tf\notin \so^{\ge \alpha+1},~\textrm{for any}~t,~1\le t\le n.
\end{equation}

Denote, for any $q$, $1\le q\le n$, $$L_q(\eta):=\partial_q(x^{\eta}x_q)f-(\alpha+k) x^{\eta}x_q\partial_qf.$$ 

We will show that: For any $i\in I_\eta$, $L_i(\eta)\in I_{1}(D)\backslash \tilde{I}_{1}(D).$

 Fix any $\eta\in E$ and denote $J_\eta=\{j_1,\ldots, j_s\}$.
 For any $i\in I_\eta$, write 
 $$L_i(\eta)=\big{(}(\eta_i+1)w_i-\alpha-k\big{)}x^\eta x_i\partial_if+(\eta_i+1)\cdot\sum_{j\neq i}w_jx^\eta x_j\partial_jf.$$
 Now we suppose $L_i(\eta)\in \tilde{I}_1(D).$ Since $x^\eta x_i\partial_if\in \tilde{I}_1(D)$ for any $i\in I_\eta$, we obtain
 $$\sum_{t=1}^sw_{j_t} x^\eta x_{j_t}\partial_{j_t}f\in \tilde{I}_1(D)=\so^{\ge\alpha+1}+(\partial_1f,\ldots,\partial_nf)\so^{\ge\alpha}.$$
 Knowing (\ref{eq2}), we consider the above sum modding out all the partial derivative of $f$ except $\partial_{j_1}f$. Then we have $$x^\eta x_{j_1}\partial_{j_1}f\equiv \partial_{j_1}f\cdot g \mod (\partial_1f,\ldots,\partial_{j_1-1}f,\partial_{j_1+1}f,\ldots, \partial_nf),~\textrm{for some}~g\in \so^{\ge\alpha}.$$
 The assumption that $f$ has an isolated singularity at the origin implies that $\{\partial_1f,\ldots, \partial_nf\}$ is a regular sequence for $\so_X$. This is possible only if $$x^\eta x_{j_1}-g\in(\partial_1f,\ldots,\partial_{j_1-1}f,\partial_{j_1+1}f,\ldots, \partial_nf)\subseteq \so^{\ge\alpha}.$$ 
 However,  this contradicts to $x^\eta x_{j_1}\notin \so^{\ge\alpha}$. So we obtain $L_i(\eta)\notin \tilde{I}_1(D)$. Clearly $L_i(\eta)\in I_1(D)$, hence $I_1(D)\neq \tilde{I}_1(D)$. 
 \end{proof}
 The following example shows that, given a weighted homogeneous polynomial with an isolated singularity at the origin, when changing the coefficient $\alpha$, $I_0(\alpha Z)$ may be equal to some power of the maximal ideal, which implies $I_1(D)=\tilde{I}_1(D)$, or equal to some ideal lacking symmetry so that $I_1(D)\neq \tilde{I}_1(D)$, which reflects the above proposition.
\begin{example}\label{Ex:notequal}
Let $f=x^2+y^5$ and $Z=\diV(f)$, $D=\alpha Z$ where $0<\alpha\le 1$. We get the weights are $w=(w_1,w_2)=(\frac{1}{2},\frac{1}{5})$. Fix a monomial basis $\{1,y,y^2,y^3\}$ of the Jacobian algebra $\fC\{x,y\}/(\partial f)$. We get $\rho(1)=\frac{7}{10}$, $\rho(y)=\frac{9}{10}$, $\rho(y^2)=\frac{11}{10}$ and $\rho(y^3)=\frac{13}{10}$.
\begin{itemize}
\item When $0<\alpha\le\frac{7}{10}$, we have $I_0(\alpha Z)=\so_X$ by Corollary \ref{C:B}. Then we see from Corollary \ref{C:B} and Proposition \ref{P:hodgemicroequality} that $I_1(D)=\tilde{I}_1(D)=(x,y^2)$ if $0<\alpha\le\frac{1}{10}$; $I_1(D)=\tilde{I}_1(D)=(x,y^3)$ if $\frac{1}{10}<\alpha\le\frac{3}{10}$; and $I_1(D)=\tilde{I}_1(D)=(x,y^4)$ if $\frac{3}{10}<\alpha\le\frac{7}{10}$.
\item When $\frac{7}{10}<\alpha\le \frac{9}{10}$, similarly we have $I_0(\alpha Z)=(x,y)$ and $I_1(D)=\tilde{I}_1(D)=(x,y^4)(x,y)$.

\item When $\frac{9}{10}<\alpha\le 1$, similarly we have $I_0(\alpha Z)=(x,y^2)$. We will see $I_1(D)\neq \tilde{I}_1(D)$. 
In fact, by Corollary \ref{C:B},
\begin{align*}
I_1(D)&=\so^{\ge\alpha+1}+\sum_{1\le i\le 2, a\in \so^{\ge \alpha}}\so_X\big{(}f\partial_ia-\alpha a\partial_if\big{)}\\
&=(x^3, x^2y^2, xy^4, y^7)+\big{(}(1-2\alpha)x^2+y^5, xy^2, 2x^2y+(2-5\alpha)y^6\big{)}\\
&=\big{(}(1-2\alpha)x^2+y^5, xy^2, 2x^2y+(2-5\alpha)y^6\big{)},\\
\end{align*} 
and  by Proposition \ref{P:saitomicrovfil}, 
\begin{align*}
\tilde{I}_1(D)&=\sum_{v_j:\Vw(v_j)\ge \alpha+1}\mathcal{O}_Xv_j+(\partial f)I_0(D)\\
&=(x,y^4)(x,y^2)=(x^2, xy^2, y^6).\\
\end{align*}
Clearly $(1-2\alpha)x^2+y^5\in I_1(D)\backslash \tilde{I}_1(D)$ and $x^2\in \tilde{I}_1(D)\backslash I_1(D)$.
\end{itemize}
\end{example}


  \section{Non-degenerate case}\label{S:nondegenerate}
  In this section, we extend Theorem \ref{thma} to the case that $D=\alpha Z$, where $0<\alpha\le 1$ and $Z=(f=0)$ is an integral and reduced effective divisor defined by $f$, a germ of holomorphic function that has non-degenerate Newton boundary with respect to the local coordinates and has an isolated singularity at the origin. The proof is very similar. One only needs to point out that a version of Lemma \ref{L:421} still holds.
 To state the results in this section, we denote:
\begin{itemize}
\item $\so=\fC\{x_1,\ldots, x_n\}$ the ring of germs of holomorphic function for local coordinates $x_1,\ldots, x_n$.

\item
$f:(\fC^n,0)\to (\fC,0)$ a germ of holomorphic function, and we write 
$$f=\sum_{A\in\rN^n}f_Ax^A,$$
where $A=(a_1,\ldots, a_n)$ and $x^A=x_1^{a_1}\cdots x_n^{a_n}$. 

\item $N(f)=\{A\in\rN^n: f_A\neq 0\}$.

\item $\Gamma=\Gamma(f)$ the union of compact faces of the convex hull of $N(f)+\rN^n$ in $(\fR^+)^n$.

\item $\mathcal{F}$ the set of faces in $\Gamma$ of dimension $n-1$.

\item $\mathcal{F}'$ the set of faces of dimension $n-1$ of the convex hull of $N(f)+\rN^n$ in $(\fR^+)^n$ that are not contained in any hyperplane $(x_i=0)$ for any $1\le i\le n$.

\item $B_F$ for $F\in\mathcal{F}'$, the unique vector of $(\fQ^+)^n$ such that
$$\langle A, B_F\rangle=1,~\forall~A\in F.$$
Denote $B_F=(b_{1,F},\cdots, b_{n,F})$ and $$\abs{B_F}=\sum_{i=1}^n b_{i,F}.$$

\item $\tilde{\rho}_F(g)$ the weight of an element $g=\sum g_Ax^A\in\so$ with respect to $F\in\mathcal{F}'$ defined by $$\tilde{\rho}_F(g)=\abs{B_F}+\inf\{\langle A, B_F\rangle: g_A\neq 0\}.$$

\item $\tilde{\rho}$ the weight of $g$ with respect to $\Gamma$ defined by
$$\tilde{\rho}(g)=\inf\{\tilde{\rho}_F(g): F\in\mathcal{F}'\}.$$
 The weight function $\tilde{\rho}$ defines a filtration on $\so$ as 
$$\tilde{\so}^{>p}=\{u\in\so: \tilde{\rho}(u)>p\};$$
$$\tilde{\so}^{\ge p}=\{u\in\so:\tilde{\rho}(u)\ge p\}.$$
\item $\hat{\rho}_F(g)$ the unshifted weight of an element $g=\sum g_Ax^A\in\so$ with respect to $F\in\mathcal{F}'$ defined by $$\hat{\rho}_F(g)=\inf\{\langle A, B_F\rangle: g_A\neq 0\}.$$

\item $\hat{\rho}$ the unshifted weight of $g$ with respect to $\Gamma$ defined by
$$\hat{\rho}(g)=\inf\{\hat{\rho}_F(g): F\in\mathcal{F}'\}.$$
Clearly, $\tilde{\rho}(g)=\hat{\rho}(gx_1x_2\ldots x_n)$.

\end{itemize}
 \begin{remark}
 Supposing $f$ is a weighted homogeneous polynomial, under the assumption that $\Gamma$ intersects with each coordinate axis at one point, i.e., for any $1\le i\le n$, there exists $a_i\in\rN$ such that $x_i^{a_i}$ has nonzero coefficient in the expression of $f$, the weight function $\tilde{\rho}$ defined as above coincide with the weight function $\rho$ defined in (\ref{E:wtfunction}); otherwise, these two weight functions may not be the same.
 \end{remark}
\begin{definition}\label{D:convandnond}
We say that $f$ is \emph{convenient} if $\Gamma$ intersects with each coordinate axis at one point. We say that $f$ has \emph{non-degenerate Newton boundary with respect to the local coordinates} if for any face $F\in\Gamma$ (of any dimension), the restriction of $f$ to $F$, i.e., $$f|_F=\sum_{A\in F}f_Ax^A$$
satisfies the condition:
$$x_1\frac{\partial f|_F}{\partial x_1}=\cdots=x_n\frac{\partial f|_F}{\partial x_n}=0~\Rightarrow~ x_1x_2\cdots x_n=0.$$
\end{definition}
We will use the following result which illustrates a property of the weight function $\tilde{\rho}$ :
\begin{lemma}\cite[Proposition B.1.2.3.(ii)*]{BGMM89}) \label{L:weightfunctiong}
If we assume $f$ is a germ of holomorphic function that has non-degenerate Newton boundary with respect to the local coordinates and convenient, then any element $g\in (\partial f)$ can be written as
$$g=g_1\partial_1f+\cdots+g_n\partial_nf$$
such that 
$$\tilde{\rho}(g_j\partial_jf)\ge \tilde{\rho}(g),$$
$$\tilde{\rho}(g_j)\ge \tilde{\rho}(g)-1+\hat{\rho}(x_j),$$
and
$$\tilde{\rho}(\frac{\partial g_j}{\partial x_j})\ge \tilde{\rho}(g)-1$$ for any $1\le j\le n$.
\end{lemma}
 By very similar methods we can extend Theorem \ref{thma} to the case when $f$ is convenient and has non-degenerate Newton boundary (see Definition \ref{D:convandnond}). One of the key results we need is the following lemma, very similar to Lemma \ref{L:421}:
\begin{lemma}
If we assume $f$ is a germ of holomorphic function that is convenient and has non-degenerate Newton boundary with respect to the local coordinates, then we have
$$F_pV^{\alpha}(\so_X[\partial_t,\partial_t^{-1}])=\sum_{i\le p}F_{p-i}\doD_{X}(\tilde{\so}^{\ge \alpha+i}\otimes\partial_t^i)\quad\textrm{for any}~\alpha\in\fQ.$$
\end{lemma}
\begin{proof}
Fix $\alpha\in\fQ$. We denote the sum on the right hand side of the equation to be $$S_p=\sum_{i\le p}F_{p-i}\doD_{X}(\tilde{\so}^{\ge \alpha+i}\otimes\partial_t^i)$$ for any $p\in\rZ$. It is enough to show
 $$S_p\cap(\sum_{i\le p-1}\so_X\otimes\partial_t^i)\subseteq S_{p-1}$$
 for any $p\in Z$. Take an arbitrary element $a$ in the set $$S_p\cap(\sum_{i\le p-1}\so_X\otimes\partial_t^i)-S_{p-1},$$i.e., 
 $$a=\sum_{\nu\in\rN^n}\partial_x^\nu(a_{\nu}\otimes\partial_t^{p-\abs{\nu}})$$
 with $a_\nu\in\tilde{\so}^{\ge \alpha+p-\abs{\nu}}$ such that 
 \begin{equation}\label{E:lhsc}
 \sum_{\nu\in\rN^n}(-1)^\nu(\partial f)^\nu a_\nu=0
 \end{equation}
where $(\partial f)^\nu:=(\frac{\partial f}{\partial x_1})^{\nu_1}\cdot(\frac{\partial f}{\partial x_2})^{\nu_2}\cdots(\frac{\partial f}{\partial x_n})^{\nu_n}$. We want to show that $a\in S_{p-1}$. 

Define  $k(a)=\min\{\abs{\nu}:a_\nu\neq 0\}$ for any such $a$ written in the above expression. We claim that for any $a\in S_p\cap (\sum_{i\le p-1}\so_X\otimes \partial_t^i)-S_{p-1}$, there exists $b\in S_{p-1}$ such that $k(a')\ge k(a)+1$ for the new element $a':=a-b$. Assume $k(a)=k\ge 0$. So
$$a=\sum_{\nu:\abs{v}\ge k} \partial_x^\nu(a_\nu\otimes \partial_t^{p-\abs{\nu}})$$
where $a_\nu\in\tilde{\so}^{\ge\alpha+p-\abs{\nu}}$ such that
\begin{equation}\label{E:lhsc2}
\sum_{\nu:\abs{\nu}\ge k}(-1)^\nu(\partial f)^{\nu}a_\nu=0.
\end{equation}
 Then (\ref{E:lhsc2}) implies that 
$$\sum_{\nu:\abs{\nu}=k}(\partial f)^\nu a_\nu\in (\partial f)^{k+1}.$$
Since $f$ is convenient and has nondegenerate Newton boundary, it is well known that $f$ has an isolated singularity at the origin (see \cite[Remark B.1.1.3]{BGMM89}). Thus we know $(\partial_1f,\cdots,\partial_nf)$ is a regular sequence. Followed by \cite[Remark (ii) after 4.2]{Sai09}, we obtain $a_{\nu}\in(\partial f)$ for $\abs{\nu}=k$. Therefore by Lemma \ref{L:weightfunctiong} we can write 
$$a_{\nu}=\sum_{i=1}^n\partial_if\cdot a_{\nu,i},~\abs{\nu}=k$$
and have the inequalities $\tilde{\rho}(a_{\nu,i})\ge \tilde{\rho}(a_{\nu})-1+\hat{\rho}(x_i)$ and $\tilde{\rho}(\partial_ia_{\nu,i})\ge \tilde{\rho}(a_\nu)-1$. So we know $\partial_x^{\nu}(\partial_ia_{\nu,i}\otimes \partial_t^{p-\abs{\nu}-1})\in F_k\doD_X(\tilde{\so}^{\ge p-1-k}\otimes \partial_t^{p-1-k})\subset S_{p-1}$ for $\nu$ with $\abs{\nu}=k$. Consider 
\begin{align*}
a'&:=a-\sum_{\nu:\abs{\nu}=k}\sum_{i=1}^n\partial_x^\nu(\partial_ia_{\nu,i}\otimes \partial_t^{p-k-1})\\
&=\sum_{\nu:\abs{\nu}=k}\sum_{i=1}^n\partial_x^\nu\big{(}\partial_ifa_{\nu,i}\otimes \partial_t^{p-k}-\partial_i(a_{\nu,i})\otimes \partial_t^{p-k-1}\big{)}+\sum_{\nu:\nu>k}\partial_x^\nu(a_\nu\otimes \partial_t^{p-\abs{\nu}})\\
&=\sum_{\nu:\abs{\nu}=k}\sum_{i=1}^n\partial_x^{\nu+\bf{1}_i}(a_{\nu+\bf{1}_i}'\otimes \partial_t^{p-k-1})+\sum_{\nu:\abs{\nu}\ge k+2}\partial_x^{\nu}(a_\nu\otimes \partial_t^{p-\abs{\nu}})
\end{align*}
where $a'_{\nu+\bf{1}_i}=a_{\nu+\bf{1}_i}-a_{\nu,i}\in \tilde{\so}^{\alpha+p-k-1}$ for $\abs{\nu}=k$. Thus $a'$ is also in the set $S_p\cap (\sum_{i\le p-1}\so_X\otimes \partial_t^i)-S_{p-1}$ and $k(a')=k+1$.
Therefore by induction, $k(a)$ can be assumed to be arbitrarily large. In particular, it is allowed to assume $k(a)=N> \alpha+p$ and $(\partial f)^\nu\in \tilde{O}^{\ge \alpha+p}$ for any $\nu$ with $\abs{\nu}=N-[\alpha]-p-1$. Consequently, we have
$$(\partial f)^{\nu-\nu'}\partial_x^{\nu'}(a_\nu)\otimes\partial_t^{p-\abs{\nu'}}\in \tilde{\so}^{\ge \alpha+p-\abs{\nu'}}\otimes \partial_t^{p-\nu'}, ~\textrm{where}~ \nu'\neq \bf{0}.$$
Observe that
\begin{equation}\label{E:subsp}
\tilde{\so}^{\ge \alpha+p-1-i}\otimes \partial_t^{p-1-i}\subset S_{p-1},~\textrm{for any}~i\ge 0.
\end{equation}
 In particular,  $\tilde{\so}\otimes \partial_t^{p-1-i}\subset S_{p-1}$ for $i\ge [\alpha]+p$.
Hence we finish the proof since (\ref{E:subsp}) implies $a\in S_{p-1}$.
\end{proof}
Then we extend the formula for the Hodge filtration to the non-degenerate case. The proof is the same as that of Theorem \ref{thma}, noting that we still have the property that $\tilde{\rho}(uf)\ge \tilde{\rho}(u)+1$ for any $u\in\so_X$.
\begin{theorem}\label{thmnewton}
If we assume $D=\alpha Z$, where $0<\alpha\le 1$ and $Z=(f=0)$ is an integral and reduced effective divisor defined by $f$, a germ of holomorphic function that is convenient and has non-degenerate Newton boundary with respect to the local coordinates, then we have
 $$F_p\dmM(f^{1-\alpha})=\sum_{i=0}^pF_{p-i}\doD_X\cdot\left(\frac{\tilde{\so}^{\ge \alpha+i}}{f^{i+1}}f^{1-\alpha}\right).$$
where the action $\cdot$ of $\doD_X$ on the right hand side of equation is the action on the left $\doD_X$-module $\dmM(f^{1-\alpha})$ defined by
$$D\cdot(wf^{1-\alpha}):=\left(D(w)+w\frac{(1-\alpha)D(f)}{f}\right)f^{1-\alpha}, ~\textrm{for any}~D\in \textrm{Der}_{\fC}\so_X.$$
\end{theorem}
\begin{remark}
Notice that Theorem \ref{thma} and Theorem \ref{thmnewton} do not imply each other. For example, $f=(x_1+x_2)^2+x_1x_3+x_3^2$ is a weighted homogeneous polynomial, but it degenerates on the face $\overline{(1,0,0)(0,1,0)}$. On the other hand, there exist many polynomials having non-degenerate Newton boundary, but which are not weighted homogeneous, e.g. $f=x^2y^2+x^5+y^5$.
\end{remark}


\begin{thebibliography}{ZZZZ99}


\bibitem[BGMM89]{BGMM89} Jo\"{e}l Brian\c{c}on, Michel Granger, Phillippe Maisonobe, and Michel Miniconi, \textit{Algorithme de calcul du polyn\^{o}me de Bernstein: Cas non d\'{e}g\'{e}n\'{e}r\'{e}}, Annales de l'institute Fourier, tome 39, no. 3(1989), pp. 553-610.

\bibitem[HTT08]{HTT08} Ryoshi Hotta, Kiyoshi Takeuchi, and Toshiyuki Tanisaki, $D$-\textit{modules, perverse sheaves, and representation theory}, Birkh\"{a}user, Boston, 2008.

\bibitem[Gra10]{Gra10} Michel Granger, \textit{Bernstein-Sato polynomials and functional equations}, Algebraic approach to differential equations, World Sci. Publ., Hackensack, NJ, 2010, pp. 225-291. 

\bibitem[Kas83]{Kas83} Masaki Kashiwara, \textit{Vanishing cycle sheaves and holonomic systems of differential equations}, Algebraic geometry (Tokyo/Kyoto, 1982), Lecture Notes in Math., vol. 1016, Springer, Berlin, 1983, pp. 134-142. 

\bibitem[Laz04]{Laz04} Robert Lazarsfeld, \textit{Positivity in algebraic geometry II},  Ergebnisse der Mathematik und ihrer Grenzgebiete, vol. 49, Springer-Verlag, Berlin, 2004.

\bibitem[Mal83]{Mal83} Bernard Malgrange, \textit{Polyn\^{o}mes de Bernstein-Sato et cohomologie \'{e}vanescente}, Analysis and topology on singular spaces, II, III (Luminy, 1981), 1983, pp. 243-267. 

\bibitem[MP16]{MP16} Mircea Musta\c{t}\u{a} and Mihnea Popa, \textit{Hodge ideals},
preprint arXiv:1605.08088, to appear in Memoirs of the AMS(2016).

\bibitem[MP18a]{MP18a} Mircea Musta\c{t}\u{a} and Mihnea Popa, \textit{Hodge ideals for}~$\mathbb{Q}$-\textit{divisors: birational approach}, preprint arXiv:1807.01932v2 (2018). 

\bibitem[MP18b]{MP18b} Mircea Musta\c{t}\u{a} and Mihnea Popa, \textit{Hodge ideals for}~$\mathbb{Q}$-\textit{divisors},~$V$\textit{-filtration, and minimal exponent}, preprint arXiv:1807.01935v3 (2018). 

\bibitem[Pop18]{Pop18} Mihnea Popa, $\mathscr{D}$-\textit{modules in birational geometry}, preprint arXiv:1807.02375, to appear in the Proceedings of the ICM, Rio de Janeiro (2018).

\bibitem[KSai71]{KSai71} Kyoji Saito, \textit{Quasihomogene isolierte Singularit\"{a}ten von Hyperfl\"{a}chen}, Inventiones Math. 14 (1971), 123-142.

\bibitem[Sai88]{Sai88} Morihiko Saito, \textit{Modules de Hodge polarisables}, Publ. Res. Inst. Math. Sci., Kyoto Univ. \textbf{24} (1988), no. 6, 849--995.

\bibitem[Sai90]{Sai90} Morihiko Saito, \textit{Mixed Hodge modules}, Publ. Res. Inst. Math. Sci., Kyoto Univ. \textbf{26} (1990), no. 2, 221--333. 

\bibitem[Sai94]{Sai94} Morihiko Saito, \textit{On microlocal}~$b$\textit{-function}, Bull. Soc. Math. France \textbf{122} (1994), no. 2, 163--184. 

\bibitem[Sai09]{Sai09} Morihiko Saito, \textit{On the Hodge filtration of Hodge modules}, Mosc. Math. J. \textbf{9} (2009), no. 1, 161--191, back matter. 

\bibitem[Sai17]{Sai16} Morihiko Saito, \textit{Hodge ideals and microlocal}~$V$\textit{-filtration}, preprint arXiv:1612.08667 (2017).

\bibitem[Sai18]{Sai18} Morihiko Saito, \textit{Personal communication}, August, 2018.
\end{thebibliography}
\end{document}